\documentclass{amsart}
\usepackage{amsmath, amsthm, amssymb, amscd, mathrsfs, eucal, epsfig}
\usepackage{pinlabel}

\begin{document}

\newtheorem{theorem}{Theorem}[section]
\newtheorem{lemma}[theorem]{Lemma}
\newtheorem{proposition}[theorem]{Proposition}
\newtheorem{corollary}[theorem]{Corollary}
\newtheorem{conjecture}[theorem]{Conjecture}
\newtheorem{question}[theorem]{Question}
\newtheorem{problem}[theorem]{Problem}
\newtheorem*{claim}{Claim}
\newtheorem*{criterion}{Criterion}
\newtheorem*{face_thm}{Theorem A}
\newtheorem*{rot_thm}{Theorem B}
\newtheorem*{immerse_thm}{Theorem C}

\theoremstyle{definition}
\newtheorem{definition}[theorem]{Definition}
\newtheorem{construction}[theorem]{Construction}
\newtheorem{notation}[theorem]{Notation}

\theoremstyle{remark}
\newtheorem{remark}[theorem]{Remark}
\newtheorem{example}[theorem]{Example}

\numberwithin{equation}{subsection}

\def\Z{\mathbb Z}
\def\N{\mathbb N}
\def\R{\mathbb R}
\def\Q{\mathbb Q}
\def\D{\mathcal D}
\def\E{\mathcal E}
\def\RR{\mathcal R}
\def\P{\mathcal P}
\def\F{\mathcal F}

\def\cl{\textnormal{cl}}
\def\scl{\textnormal{scl}}
\def\homeo{\textnormal{Homeo}}
\def\rot{\textnormal{rot}}
\def\area{\textnormal{area}}

\def\Id{\textnormal{Id}}
\def\SL{\textnormal{SL}}
\def\Sp{\textnormal{Sp}}
\def\PSL{\textnormal{PSL}}
\def\length{\textnormal{length}}
\def\fill{\textnormal{fill}}
\def\rank{\textnormal{rank}}
\def\til{\widetilde}

\title{Faces of the scl norm ball}
\author{Danny Calegari}
\address{Department of Mathematics \\ Caltech \\
Pasadena CA, 91125}
\email{dannyc@its.caltech.edu}

\date{1/22/2008, Version 0.09}

\begin{abstract}
Let $F = \pi_1(S)$ where $S$ is a compact, connected, oriented surface with $\chi(S)<0$ and nonempty
boundary.
\begin{enumerate}
\item{The projective class of the the chain $\partial S \in B_1^H(F)$ intersects the interior of a
codimension one face $\pi_S$ of the unit ball in the stable commutator length norm on $B_1^H(F)$.}
\item{The unique homogeneous quasimorphism on $F$ dual to $\pi_S$ (up to scale and elements of $H^1(F)$) is the rotation
quasimorphism associated to the action of $\pi_1(S)$ on the ideal boundary of the hyperbolic plane,
coming from a hyperbolic structure on $S$.}
\end{enumerate}
These facts follow from the fact that every homologically trivial $1$-chain $C$ in $S$
rationally cobounds an immersed surface with a sufficiently large
multiple of $\partial S$. This is true even if $S$ has no boundary.
\end{abstract}

\maketitle

\section{Introduction}

An immersed loop in the plane might or might not bound an immersed disk, and if it
does, the disk it bounds might not be unique. An immersed loop on a surface might
not bound an immersed subsurface, but admit a finite cover which does bound
an immersed subsurface --- i.e. it might ``virtually'' bound an immersed surface. 
Most homologically trivial geodesics on hyperbolic surfaces with boundary do not 
even virtually bound an immersed surface. However, in this paper, we show that
{\em every} homologically trivial geodesic in a {\em closed} hyperbolic surface $S$ virtually
bounds an immersed surface, and every homologically trivial geodesic in a hyperbolic
surface $S$ with boundary virtually cobounds an immersed surface together with a
sufficiently large multiple of $\partial S$. This has implications for the structure
of the (second) bounded cohomology of free and surface groups, as we explain in what follows.

\vskip 12pt

Given a group $G$ and an element $g \in [G,G]$, the {\em commutator length} of $g$,
denoted $\cl(g)$, is the smallest number of commutators in $G$ whose product is $g$,
and the {\em stable commutator length} of $g$ is the limit $\scl(g):=\lim_{n \to \infty} \cl(g^n)/n$.
Geometrically, if $X$ is a space with $\pi_1(X)=G$ and $g$ is represented by a loop 
$\gamma$ in $X$, the commutator length of $g$ is the least genus of a surface mapping to
$X$ whose boundary maps to $\gamma$. By minimizing number of triangles instead of genus, one can
reinterpret $\scl$ as a kind of $L^1$ norm on relative ($2$-dimensional) homology.
Technically, if $B_1(G;\R)$ denotes the vector space of real-valued (group) $1$-boundaries (i.e.
group $1$-chains which are boundaries of group $2$-chains; see \S~\ref{pseudo_norm_subsection}),
there is a well-defined $\scl$ pseudo-norm on $B_1(G;\R)$. The subspace on which $\scl$ vanishes
always includes a subspace $H$ spanned by chains of the form $g^n - ng$ and $g - hgh^{-1}$, 
for $n \in \Z$ and $g,h \in G$, and therefore $\scl$ descends to a pseudo-norm on the
quotient $B_1(G;\R)/H$, which we abbreviate by $B_1^H(G;\R)$ or $B_1^H(G)$
or even $B_1^H$ in the sequel. In certain special cases (for example, when $G$ is
a free group), $\scl$ defines an honest norm on $B_1^H(F)$, but we will not use this fact
in the sequel. More precise definitions and background are given in \S~\ref{background_section}.

Dual (in a certain sense) to the space $B_1^H(G;\R)$ with its $\scl$ pseudo-norm, is the
space $Q(G)$ of {\em homogeneous quasimorphisms} on $G$, i.e. functions $\phi:G \to \R$ 
for which there is a least real number $D(\phi)$ (called the {\em defect}) such that
 $\phi(g^n) = n\phi(g)$ for all $g \in G$ and $n \in \Z$, 
and $|\phi(g) + \phi(h) - \phi(gh)| \le D(\phi)$ for
all $g,h \in G$. The particular form
of duality between $\scl$ and $Q(G)$ is called {\em Bavard duality}, which is the equality
$$\scl(\sum t_i g_i) = \sup_{\phi \in Q(G)} \frac {\sum t_i \phi(g_i)} {2D(\phi)}$$
(see \S~\ref{background_section} for details).

\vskip 12pt

The defining properties of a homogeneous quasimorphism can be thought of as an infinite
family of linear equalities and inequalities depending on elements and pairs of
elements in $G$. The $L^1$-$L^\infty$ duality between $\scl$ and $Q(G)$ means that computing $\scl$
is tantamount to solving an (infinite dimensional) {\em linear programming problem} in
group homology (for an introduction to linear programming, see e.g. \cite{Dantzig}). 
In finite dimensions, $L^1$ and $L^\infty$ norms are piecewise linear functions,
and their unit balls are rational convex polyhedra. Broadly speaking, the main discovery of \cite{Calegari_pickle}
is that in free groups (and certain groups derived from free groups in simple ways), 
computing $\scl$ reduces to a {\em finite dimensional} linear programming problem, and
therefore the unit ball of the $\scl$ pseudo-norm on $B_1^H(F;\R)$ is a rational convex polyhedron;
i.e. for every finite dimensional rational vector subspace $V$ of $B_1(F;\R)$, the
unit ball of the $\scl$ pseudo-norm restricted to $V$ is a finite-sided rational convex polyhedron
(compare with the well-known example of the Gromov-Thurston norm on $H_2$ of a $3$-manifold;
see \cite{Thurston_norm}).

\vskip 12pt

In a finite dimensional vector space, a rational convex polyhedron is characterized by its top
dimensional faces --- i.e. those which are codimension one in the ambient space.
In an {\em infinite} dimensional vector space, a rational convex polyhedron
need not have any faces of finite codimension at all. The codimension of a face of a convex 
polyhedron in an infinite dimensional vector space is the supremum of the codimensions of
its intersections with finite dimensional subspaces.
Top dimensional faces of the unit ball of the Gromov-Thurston norm on $H_2$ of a $3$-manifold have a
great deal of topological significance (see e.g \cite{Thurston_norm}, \cite{Kronheimer_Mrowka}, 
\cite{Calegari_quasi} or \cite{Oszvath_Szabo} for connections with the theories of 
taut foliations, Seiberg-Witten equations, quasigeodesic flows,
and Heegaard Floer homology respectively). It is therefore a natural question to ask whether the
$\scl$ unit polyhedron in $B_1^H(F;\R)$ has any faces which are codimension one in
$B_1^H(F;\R)$, and whether some of these faces have any geometric significance. Our
first two main theorems answer these questions affirmatively. 

\begin{face_thm}
Let $F$ be a free group, and let $S$ be a compact, connected, orientable surface with
$\chi(S)<0$ and $\pi_1(S) = F$. Let $\partial S \in B_1^H(F;\R)$ be the $1$-chain represented
by the boundary of $S$, thought of as a finite formal sum of conjugacy classes in $F$.
Then the projective ray in $B_1^H(F;\R)$ spanned by $\partial S$ intersects the
unit ball of the $\scl$ norm in the interior of a face of codimension one in $B_1^H(F;\R)$.
\end{face_thm}

By Bavard duality, a face of the $\scl$ norm ball of codimension one is dual to
a unique extremal homogeneous quasimorphism, up to elements of $H^1$ (which vanish
identically on $B_1^H$). It turns out that we can give an explicit description of
the extremal quasimorphisms dual to the ``geometric'' faces of the $\scl$ norm ball described in
Theorem~A.

If $S$ is a compact, connected, orientable surface with
$\chi(S)<0$, then $S$ admits a hyperbolic structure with geodesic boundary. The
hyperbolic structure and a choice of orientation determine a discrete, faithful
representation $\rho:\pi_1(S) \to \PSL(2,\R)$, unique up to conjugacy. Since $\pi_1(S)$
is free, this representation lifts to $\til{\rho}:\pi_1(S) \to \til{\SL}(2,\R)$ where
$\til{\SL}(2,\R)$ denotes the universal covering group of $\PSL(2,\R)$.
There is a unique continuous homogeneous quasimorphism on $\til{\SL}(2,\R)$ (up to scale), called the
{\em rotation quasimorphism} (discussed in detail in \S~\ref{rotation_subsection}). 
This quasimorphism pulls back by $\til{\rho}$ to a homogeneous quasimorphism $\rot_S$ on
$\pi_1(S)$, which is well-defined up to elements of $H^1(S;\Z)$. Up to scale, this turns
out to be the homogeneous quasimorphism dual to the top dimensional face of the $\scl$
norm ball described above:

\begin{rot_thm}
Let $F$ be a free group, and let $S$ be a compact, connected, orientable surface
with $\chi(S)<0$ and $\pi_1(S) = F$. Let $\pi_S$ be the face of the $\scl$ unit norm ball
whose interior intersects the projective ray of the class $\partial S$. The face
$\pi_S$ is dual to the extremal homogeneous quasimorphism $\rot_S$.
\end{rot_thm}

These theorems are both proved in \S~3.

Theorem~A shows how hyperbolic geometry
and surface topology manifest in the abstract (bounded) cohomology of a free
group. Theorem~B is a kind of rigidity result, characterizing the rotation quasimorphism associated to a
discrete, faithful representation of $\pi_1(S)$ into $\PSL(2,\R)$ amongst all homogeneous quasimorphisms
by the property that it is ``extremal'' for $\partial S$. In \S~\ref{remark_corollary_subsection} 
we use these theorems to deduce a short proof of a relative version of
rigidity theorems of Goldman \cite{Goldman} and Burger-Iozzi-Wienhard
\cite{Burger_Iozzi_Wienhard}, that representations of surface groups into
certain Lie groups of maximal Euler class are discrete (it should be stressed that
\cite{Goldman, Burger_Iozzi_Wienhard} contain much more than the narrow result
we reprove).

In light of Theorem~A and Theorem~B, it is natural to ask whether the projective
class of every element $g \in [F,F]$ intersects the interior of a face of the
$\scl$ norm ball of finite codimension. In fact, it turns out that this is not the case.
We show by an explicit example (Example~\ref{infinite_codimension}) that there
are many elements $g \in [F,F]$ where $F$ has rank at least $4$, whose projective classes
are contained in faces of the $\scl$ norm ball of infinite codimension.

\vskip 12pt

The method of proof is of independent interest. We show that for any homologically
trivial geodesic $1$-manifold $\gamma$ in a hyperbolic surface $S$, there is a surface $T$
and an immersion $f:T,\partial T \to S,\gamma$ for which $f_*[\partial T]$ is taken to
some multiple of $[\gamma] + n[\partial S]$ in $H_1$; i.e. the $1$-cycle $\gamma + n\partial S$ 
``rationally bounds'' an immersed surface. Note that this remains true even if $S$ is closed! 
Explicitly, the statement of the main technical theorem (proved in \S~\ref{immersion_theorem_subsection}) 
is as follows:

\begin{immerse_thm}
Let $S$ be a compact, connected orientable surface with $\chi(S)<0$, and $C = \sum r_i g_i$ a finite
rational chain in $B_1^H(\pi_1(S))$. Then for all sufficiently large rational
numbers $R$ (depending on $C$), the geodesic $1$-manifold in $S$ corresponding to the
chain $R\partial S + \sum r_i g_i$ rationally bounds
an immersed surface $f:T \to S$.
\end{immerse_thm}

The connection with stable commutator length is as follows: from the main theorem of \cite{Calegari_pickle}
it follows that in an oriented hyperbolic surface $S$ with boundary,
a rational $1$-chain $C$ bounds an immersed surface if and only if $\scl(C) = \rot_S(C)/2$,
where $\rot_S$ is as above (this is Proposition~\ref{rot_is_scl} below). Hence Theorem~C implies
that every chain $C$ in $B_1^H(\pi_1(S))$ which is projectively close enough to $\partial S$
satisfies $\scl(C) = \rot_S(C)/2$; Theorems A and B follow.

\vskip 12pt

A number of additional corollaries are stated, including a generalization of the main
theorem of \cite{Calegari_surface}. Let $G$ be a graph of free or (closed, orientable)
hyperbolic surface groups amalgamated over infinite cyclic subgroups, 
and let $A$ be a nonzero rational class in $H_2(G)$.
Let $[S_1],\cdots,[S_m]$ be the fundamental classes in $H_2$ of the vertex
subgroups which are closed surface groups. Then for all sufficiently big integers $n_i$,
some multiple of the class $A + \sum_i n_i\cdot[S_i]$ in $H_2(G)$ is represented by an injective map from a
closed hyperbolic surface group to $G$.

\section{Background}\label{background_section}

\subsection{Definitions}\label{definition_subsection}

The following definition is standard; see \cite{Bavard} or \cite{Calegari_scl}, \S~2.1.

\begin{definition}
Let $G$ be a group, and $g \in [G,G]$. The {\em commutator length} of $g$, denoted
$\cl(g)$, is the smallest number of commutators in $G$ whose product is $g$. 
\end{definition}

Topologically, if $X$ is a space with $\pi_1(X) = G$, and $\gamma:S^1 \to X$ is
a loop representing the conjugacy class of $g$, then $\cl(g)$ is the least genus of
a compact oriented surface $S$ with one boundary component for which there is a map
$f:S \to X$ with $f|_{\partial S}$ in the free homotopy class of $\gamma$.

\begin{definition}
Let $G$ be a group, and $g \in [G,G]$.
The {\em stable commutator length} of $g$, denoted $\scl(g)$, is the limit
$$\scl(g) = \lim_{n \to \infty} \frac {\cl(g^n)} {n}$$
\end{definition}

Commutator length and stable commutator length can be extended to finite linear sums of
groups elements as follows:
\begin{definition}\label{scl_sum_definition}
Let $G$ be a group, and $g_1,g_2,\cdots,g_m$ elements of $G$ whose product is in $[G,G]$.
Define
$$\cl(g_1 + \cdots + g_m) = \inf_{h_i \in G} \cl(g_1h_1g_2h_1^{-1}\cdots h_{m-1}g_mh_{m-1}^{-1})$$
and
$$\scl(g_1 + \cdots + g_m) = \lim_{n \to \infty} \frac {\cl(g_1^n + \cdots + g_m^n)} {n}$$
\end{definition}

A geometric interpretation of these numbers will be given in \S~\ref{extremal_surfaces_section}.

It is a fact that the limit in Definition~\ref{scl_sum_definition}
exists, and satisfies $\scl(g^n + \sum g_i) = \scl(ng + \sum g_i)$
and $\scl(g + g^{-1} + \sum g_i ) = \scl(\sum g_i)$ for all $g,g_i \in G$ and $n \in \N$. So
it makes sense to define $\scl(\sum n_i g_i) = \scl(\sum g_i^{n_i})$ for any $n_i \in \Z$ and $g_i \in G$.
With this definition, it is immediate that $\scl$ is subadditive; i.e.
$$\scl\Bigl(\sum g_i + \sum h_j\Bigr) \le \scl\Bigl(\sum g_i\Bigr) + \scl\Bigl(\sum h_j\Bigr)$$

If we need to emphasize the group $G$ we denote $\cl$ and $\scl$ by $\cl_G$ and $\scl_G$
respectively. 

\vskip 12pt

One may reduce the calculation of $\scl$ on finite sums to a calculation of
$\scl$ on individual elements, by the following ``addition lemma":

\begin{lemma}[Addition Lemma]\label{addition_lemma}
Let $g_1, \cdots, g_m \in G$ have infinite order. Let $H = G* F_{m-1}$ where
$F_{m-1}$ is freely generated by $x_1,x_2,\cdots,x_{m-1}$. Then
$$\scl_G(g_1 + \cdots + g_m) + \frac {m-1} 2 = \scl_H(g_1x_1g_2x_1^{-1}\cdots x_{m-1}g_mx_{m-1}^{-1})$$
\end{lemma}
This follows from \cite{Calegari_scl}, Thm.~2.93 and induction. When the $g_i$
have finite order, the formula must be corrected (in a straightforward way).

\subsection{scl as a pseudo-norm}\label{pseudo_norm_subsection}

It is convenient to use the language of homological algebra. Given a group $G$,
one has the complex of real group chains $C_*(G;\R),\partial$ whose homology is the
real (group) homology of $G$; see Mac Lane (\cite{Maclane}, Ch.~IV, \S~5). A
real (group) $n$-chain is a finite formal real linear combination of elements of $G^n$,
so (for instance) a real (group) $1$-chain is just a finite formal real linear combination
of elements of $G$. Denote the group of (real) $1$-boundaries by $B_1(G;\R)$, or
$B_1(G)$ for short.

The properties of $\scl$ enumerated in \S~\ref{definition_subsection} imply that the
function $\scl$ is well-defined, linear and subadditive on finite
integral group $1$-boundaries, and therefore admits a unique linear continuous extension
to $B_1(G)$.

Moreover, $\scl$ vanishes on the subspace $H$ of $B_1$ spanned by chains of the form $g^n - ng$ and
$g - hgh^{-1}$ for any $g,h \in G$ and $n \in \Z$. Thus $\scl$ defines
a pseudo-norm on $B_1^H:=B_1/H$. See \cite{Calegari_scl}, \S~2.6 for proofs
of these basic facts.

\vskip 12pt

In \cite{Calegari_pickle} an algorithm is described to compute $\scl$ on elements of
$B_1^H(F)$ where $F$ is a free group. The program {\tt scallop} (source available at \cite{scallop}) implements
a polynomial-time version of this algorithm, described in \cite{Calegari_scl}, \S~4.1.7--4.1.8. 
At a number of points in this paper we make assertions
about the value of $\scl$ on certain chains in $B_1^H(F)$; these assertions are justified
using the program {\tt scallop}.

\begin{remark}
In the final analysis, our main theorems do not depend logically on the computations carried
out with the aid of {\tt scallop} (also see Remark~\ref{explicit_surface_remark}). 
Nevertheless, these computations were an essential part of the process by which these theorems were
discovered.
\end{remark} 

\subsection{Extremal surfaces}\label{extremal_surfaces_section}

The definition of stable commutator length can be reinterpreted in geometric terms.
Let $X$ be a space, and $\gamma_1,\cdots,\gamma_m:S^1 \to X$ nontrivial free homotopy
classes of loops in $X$. Let $f:S \to X$ be a surface for which there is a commutative
diagram
$$ \begin{CD}
\partial S @>i>> S \\
@V\partial fVV	@VVfV \\
\coprod_i S^1 @>\coprod_i\gamma_i>> X 
\end{CD}$$
where $i:\partial S \to S$ is the inclusion map, and $\partial f_*[\partial S] = n[\coprod_i S^1]$ in $H_1$
for some $n$.
  
Define $\chi^-(S)$ to be the sum of Euler characteristic $\chi$ over all components of $S$
for which $\chi$ is non-positive. Then there is an equality
$$\scl_G(g_1 + \cdots + g_m) = \inf_{S} \frac {-\chi^-(S)} {2n}$$
over all compact oriented surfaces $S$ as above, where $G = \pi_1(X)$, and $\gamma_i$
represents the conjugacy class of $g_i$. See \cite{Calegari_scl}, Prop.~2.68 for a proof.

\begin{definition}\label{geometric_scl_definition}
The chain $g_1 + \cdots + g_m$ is said to 
{\em rationally bound} a surface $f:S \to X$ for which there is a commutative diagram as above. 
A surface with this property is {\em extremal} if every component of $S$ has negative
Euler characteristic, and there is equality
$$\scl(g_1 + \cdots + g_m) = \frac {-\chi^-(S)} {2n}$$
\end{definition}
Extremal surfaces are $\pi_1$-injective (\cite{Calegari_scl}, Prop.~2.96).

\subsection{Quasimorphisms}\label{duality_subsection}

\begin{definition}
Let $G$ be a group. A function $\phi:G \to \R$ is a {\em homogeneous quasimorphism}
if it satisfies $\phi(g^n) = n\phi(g)$ for all $g \in G$ and $n \in \Z$, and if
there is a least non-negative real number $D(\phi)$ called the {\em defect}, such that
for all $g,h \in G$ there is an inequality
$$|\phi(g) + \phi(h) - \phi(gh)| \le D(\phi)$$
\end{definition}

The set of homogeneous quasimorphisms on a group $G$ is a real vector space, and is
denoted $Q(G)$. A homogeneous quasimorphism has defect $0$ if and only if it is a homomorphism.
Thus $H^1(G;\R)$ is a vector subspace of $Q(G)$.
The defect $D$ defines a natural norm on $Q(G)/H^1(G;\R)$, making it into a Banach space.
Any real-valued function on $G$ extends by linearity to define a $1$-cochain. There
is an exact sequence
$$0 \to H^1(G;\R) \to Q(G) \xrightarrow{\delta} H^2_b(G;\R) \to H^2(G;\R)$$
where $\delta$ is the coboundary map on $1$-cochains, and
$H^*_b$ denotes bounded (group) cohomology. See \cite{Bavard} or \cite{Calegari_scl}, \S~2.4
for an explanation of these facts,
and Gromov \cite{Gromov_bounded} for an introduction to bounded cohomology.

There is a duality between stable commutator length and quasimorphisms, called {\em Bavard duality}.
For chains in $B_1^H(G)$, this takes the following form

\begin{theorem}[Bavard duality]\label{bavard_duality_theorem}
Let $C = \sum t_i g_i$ be an element of $B_1^H(G;\R)$. Then there is an equality
$$\scl(C) = \sup_{\phi \in Q(G)/H^1(G)} \frac {\sum t_i \phi(g_i)} {2D(\phi)}$$
\end{theorem}
See \cite{Bavard} for a proof when $C$ is an element of $[G,G]$, or
\cite{Calegari_scl}, Thm.~2.73 for the general case.

\begin{definition}
A quasimorphism $\phi \in Q(G)$ is {\em extremal} for a chain $C= \sum t_i g_i \in B_1^H(G;\R)$
if there is equality
$$\scl(C) = \frac {\sum t_i \phi(g_i)} {2D(\phi)}$$
\end{definition}

Given a chain $C$, the set of extremal quasimorphisms for $C$ is a nonempty convex cone, and is
closed (away from $0$) in the natural Banach space topology on $Q(G)$, as well as in the topology of
termwise convergence in $\R^G$. See \cite{Calegari_scl}, Prop.~2.81 for a proof.

\section{Immersed surfaces}\label{immersed_surfaces_section}

\subsection{Doodles}

The question of which immersed loops in surfaces bound immersed subsurfaces is subtle and interesting,
and has fascinated many mathematicians (see e.g. \cite{Grothendieck}, p.~47).
An immersed loop in the plane which bounds an immersed disk necessarily has
winding number $\pm 1$, but not every loop with winding number $\pm 1$ bounds
an immersed disk. See Figure~\ref{two_bad_figure}.

\begin{figure}[htpb]
\labellist
\small\hair 2pt
\endlabellist
\centering
\includegraphics[scale=0.35]{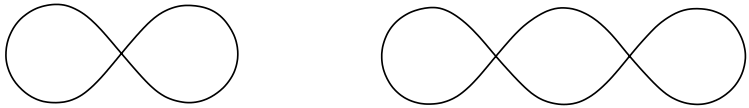}
\caption{A figure $8$ has winding number $0$, and therefore cannot bound an
immersed disk. But a ``double $8$'' has winding number $1$, and does not bound
an immersed disk either} \label{two_bad_figure}
\end{figure}

Blank \cite{Blank} gave an algorithm to determine which immersed loops in the plane
bound immersed disks; his algorithm was extended to other surfaces by Francis \cite{Francis}
and others, but the answer is not very illuminating: some curves bound immersed disks, some
don't, and the reason is complicated. Other authors have studied the existence of {\em branched}
immersions with prescribed boundary, which are much easier to construct.

Milnor \cite{Milnor_doodle} gave a well-known example of an immersed loop in the plane which bounds
two different immersed disks. See Figure~\ref{milnor_figure}.

\begin{figure}[htpb]
\labellist
\small\hair 2pt
\endlabellist
\centering
\includegraphics[scale=0.35]{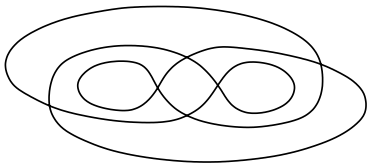}
\caption{Milnor's ``doodle'' bounds an immersed disk in two inequivalent ways} \label{milnor_figure}
\end{figure}

On a hyperbolic surface $S$, every homotopy class of essential loop contains a
canonical geodesic representative. One can therefore ask which conjugacy classes
in $\pi_1(S)$ are represented by geodesics which bound immersed surfaces. The answer
turns out to be independent of the choice of hyperbolic structure on $S$, and therefore
in principle is a purely ``algebraic'' problem.

One subtle aspect of the problem is illustrated by the example in Figure~\ref{virtual_bound_figure}. This shows
an immersed loop $\gamma$ (in the isotopy class of a geodesic)
in a genus $2$ punctured surface which does not bound an immersed surface,
but which ``virtually'' bounds an immersed surface: there is an immersed surface with
two boundary components, each of which wraps positively once around $\gamma$.

\begin{figure}[htpb]
\labellist
\small\hair 2pt
\pinlabel $\gamma$ at 220 75
\endlabellist
\centering
\includegraphics[scale=0.35]{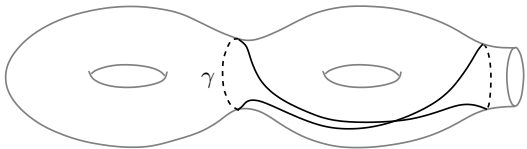}
\caption{The loop $\gamma$ does not bound an immersed surface, but two copies of
$\gamma$ do} \label{virtual_bound_figure}
\end{figure}

One is therefore led to study the following question. If $S$ is an orientable hyperbolic surface,
which conjugacy classes $g$ in $\pi_1(S)$ are represented by geodesics which virtually
bound immersed surfaces? That is, when is there an immersed surface in $S$,
all of whose boundary components wrap positively around the geodesic representative
of $g$? Only a homologically trivial loop virtually bounds a surface
at all, so we restrict attention to $g$ in the commutator subgroup.
It turns out that we can give a complete and somewhat surprising
answer to this question. If $S$ is closed, then {\em every} homologically trivial geodesic
virtually bounds an immersed surface. If $S$ has (geodesic) boundary, then a homologically
trivial geodesic corresponding to a conjugacy
class $g$ virtually bounds an immersed surface if and only if the projective class of
$g$, thought of as a $1$-boundary, intersects a certain top dimensional face of 
the unit ball in the $\scl$ pseudo-norm on $B_1^H(\pi_1(S))$. This is explained in
the remainder of this section.

\subsection{Positive immersed surfaces}

\begin{definition}
An immersion $f:T \to S$ between oriented surfaces is {\em positive} if it is orientation-preserving.
\end{definition}

\begin{definition}\label{1_manifold_bounds_definition}
Let $\gamma:\coprod_i S^1 \to S$ be an immersed $1$-manifold in $S$. The $1$-manifold
$\gamma$ {\em bounds} a (positive) immersion
$f:T \to S$ if there is a commutative diagram
$$ \begin{CD}
\partial T @>i>> T \\
@V\partial fVV	@VVfV \\
\coprod_i S^1 @>\gamma>> S
\end{CD}$$
for which $\partial f:\partial T \to \coprod_i S^1$ is an orientation-preserving
homeomorphism. The $1$-manifold $\gamma$ {\em rationally bounds} a positive immersion $f:T \to S$ 
if there is some integer $n$, and a commutative
diagram as above, for which $\partial f:\partial T \to \coprod_i S^1$ is a positive immersion
(i.e. an orientation-preserving covering map) such that $\partial f_*[\partial T] = n[\coprod_i S^1]$
in $H_1$.
\end{definition}
Compare with Definition~\ref{geometric_scl_definition}.

We are concerned in the sequel with the case that $S$ is compact, connected and oriented, possibly with
boundary, and satisfying $\chi(S) < 0$. The surface $S$ admits a (nonunique)
hyperbolic structure with totally geodesic boundary; we fix such a structure.
Let $C= \sum n_i g_i$ be a chain in $B_1^H(\pi_1(S);\R)$ where the $n_i$ are integers, and the
$g_i$ are primitive.

For each $i$, let $\gamma_i:S^1 \to S$ be a geodesic loop corresponding to the conjugacy class
of $g_i$, and let $\gamma:\coprod_i S^1 \to S$ be the union of the $\gamma_i$. We say that
$C$ rationally bounds a positive immersed surface if there is an integer $n$ and a positive
immersion $f:T \to S$ as in Definition~\ref{1_manifold_bounds_definition} for which
$\partial f_*[\partial T] = n \cdot (\sum_i n_i [S^1_i])$ in $H_1$, where $\gamma_i:S^1_i \to S$.

\begin{lemma}\label{addition_is_transitive}
Suppose chains $C_1,C_2$ rationally bound positive immersed surfaces. Then $C_1 + C_2$
rationally bounds a positive immersed surface.
\end{lemma}
\begin{proof}
For ``generic'' chains $C_1$ and $C_2$ there is nothing to prove: the disjoint union of two
immersed surfaces is an immersed surface.
The issue is that there might be a conjugacy class $g$ in the support of both $C_1$ and $C_2$
whose coefficients have different signs. Let $\gamma$ be the geodesic in the free
homotopy class corresponding to $g$. Positive immersed surfaces with rational
boundary $C_1$ and $C_2$ might have boundary components mapping to $\gamma$ with
different degrees. The following claim shows that we can construct suitable covers
of these immersed surfaces such that the boundary components mapping to $\gamma$ can be
glued up.

\begin{claim}
Let $S$ be a connected, oriented surface with $\chi(S)<0$ and genus at least $1$. Let
$\delta_S \subset \partial S$ be a set of boundary
components, and let $f_S:\delta_S \to \gamma$ be an immersion,
such that the degree of $f_S$ on every component is positive.
Let $n_i$ be the degrees of $f_S$ on the components of $\delta_S$, and let $N$ be a 
common multiple of the $n_i$. Then there is a finite cover $\widehat{S}$
such that every component of the preimage $\widehat{\delta}_S$ maps to $\gamma$ with degree $N$.
\end{claim}
\begin{proof}
An orientable surface $S$ with genus at least $1$ admits a double cover $S'$ such that every
component of $\partial S$ has exactly two preimages in $S'$. Let $\delta_{S,i}$ be
the components of $\delta_S$, and let $n_i$ be the degrees of the map $f_T:\delta_{S,i} \to \gamma$.
Let $N$ be a common multiple of the $n_i$. Define a homomorphism $\phi:\pi_1(S') \to \Z/N\Z$
as follows. For each $i$, let $\epsilon_{i,1},\epsilon_{i,2}$ be the components of $\partial S'$
in the preimage of $\delta_{S,i}$, and let $n_i$ be the degree of $f_S:\delta_{S,i} \to \gamma$.
Then define $\phi(\epsilon_{i,1}) = n_i$ and $\phi(\epsilon_{i,2}) = -n_i$. Since
$\phi(\partial S') = 0$ in $\Z/N\Z$, the function $\phi$ extends to a (surjective)
homomorphism from $\pi_1(S')$ to $\Z/N\Z$. If $\widehat{S}$ is the cover corresponding to the kernel
of $\phi$, then every component of $\widehat{\delta}_S$, the preimage of $\delta_S$,
maps to $\gamma$ with degree $N$, as desired.
\end{proof}

Start with a pair of positive immersed surfaces with rational boundary $C_1$ and $C_2$.
Since the Euler characteristics of these surfaces are negative, they admit finite covers
with genus at least one. After passing to a suitable cover (provided by the claim), and gluing up
pairs of boundary components which map to geodesics $\gamma$ in the common support
of $C_1$ and $C_2$ with the same absolute degree but with
opposite signs, we obtain a positive immersed surface which $C$ rationally bounds.
\end{proof}

\begin{remark}
Compare with the proof of Thm.~3.4 of \cite{Calegari_surface}.
\end{remark}

We will give a much shorter proof of this Lemma (assuming more machinery)
in \S~\ref{rotation_subsection}
in the special case that the ambient surface $S$ has boundary.

\subsection{Rotation quasimorphism}\label{rotation_subsection}

Throughout this section we fix $S$, a compact oriented hyperbolic surface with boundary.

Let $\gamma$ be a homologically trivial geodesic in $S$. The geodesic $\gamma$
cuts $S$ up into connected regions $R_i$. For each $i$, let $\alpha_i$ be an oriented
arc from $\partial S$ to $R_i$ which is transverse to $\gamma$, and let $n_i$
be the signed intersection of $\alpha_i$ with $\gamma$. Since $\gamma$ is
homologically trivial by hypothesis, $n_i$ does not depend on the choice of $\alpha_i$.
Geometrically, if $T$ is an oriented surface, and $f:T,\partial T \to S,\gamma$ is a smooth map, $n_i$
is a signed count of the preimages of a generic point in $R_i$.

\begin{definition}
The {\em algebraic area} enclosed by $\gamma$ is the sum
$$\area(\gamma) =  \sum n_i \cdot \area(R_i)$$
\end{definition}

The hyperbolic structure and the orientation on $S$ determines a discrete faithful
representation $\rho:\pi_1(S) \to \PSL(2,\R)$ unique up to conjugacy. Since $\pi_1(S)$
is free, this representation lifts to $\til{\rho}:\pi_1(S) \to \til{\SL}(2,\R)$, where
$\til{\SL}(2,\R)$ denotes the universal covering group of $\PSL(2,\R)$. The group $\PSL(2,\R)$
acts on the circle at infinity of hyperbolic space, and lets us think of
$\PSL(2,\R)$ as a subgroup of $\homeo^+(S^1)$. The covering group $\til{\SL}(2,\R)$ is
the preimage of $\PSL(2,\R)$ in $\homeo^+(\R)$.

\begin{definition}
Given $g \in \til{\SL}(2,\R)$, the {\em rotation number} of $g$, as defined by
Poincar\'e, is the limit
$$\rot(g) = \lim_{n \to \infty} \frac {g^n(0)} n$$
where we think of $\til{\SL}(2,\R)$ as a subgroup of $\homeo^+(\R)$ under the
covering projection $\R \to S^1 = \R/\Z$.
\end{definition}

Rotation number pulls back by $\til{\rho}$ to define a function $\rot$ on $\pi_1(S)$.
As is well-known, $\rot$ is a homogeneous quasimorphism on $\pi_1(S)$ with $D(\rot)=1$.
As a function on $\pi_1(S)$, the function $\rot$ depends on the choice of lift
of $\rho$ to $\til{\rho}$. Different lifts are classified by elements of $H^1(S;\Z)$,
so $\rot$ is well-defined on $\pi_1(S)$ modulo elements of $H^1(S;\Z)$, and therefore
well-defined on the commutator subgroup of $\pi_1(S)$ independent of the choice of
$\til{\rho}$. Though it appears to depend on the choice of hyperbolic structure on $S$, it depends
only on the topology of $S$. If we need to stress the dependence of $\rot$ on $S$,
we write it $\rot_S$.

\begin{lemma}[Area is rotation number]\label{area_rotation_number}
If $g \in \pi_1(S)$, and $\gamma$ is a geodesic in $S$ corresponding to the conjugacy
class of $g$, there is an equality
$$\area(\gamma) = 2\pi \cdot \rot_S(g)$$
\end{lemma}
This is proved in \cite{Calegari_scl}, Lem.~4.58. For the expert, the lemma follows
from the fact that the coboundary of the rotation quasimorphism 
is the (relative) Euler class, and the Gauss-Bonnet theorem.

Note that for every
nontrivial $g \in \pi_1(S)$, the element $\rho(g)$ is hyperbolic in $\PSL(2,\R)$,
and therefore fixes two points in the circle at infinity. It follows that
$\rot(g)$ is an {\em integer}. An explicit formula for $\rot(g)$, in terms
of an expression of $g$ as a reduced word in a standard generating set for
$\pi_1(S)$, is given in \cite{Calegari_scl}, Thm.~4.62.

The functions $\area$ and $\rot$ extend by linearity and continuity 
to elements in $B_1^H(\pi_1(S);\R)$. This is obvious for the function $\rot$, and
straightforward for $\area$: if $C = \sum_i t_i g_i$ where the $t_i$ are real numbers
and the $g_i$ are primitive conjugacy classes, let $\gamma_i$ be oriented geodesics in $S$ in the
conjugacy classes of the $g_i$. The $\gamma_i$ cut up $S$ into regions $R_j$. For each $j$,
let $\alpha$ be an arc from $\partial S$ to $R_j$ transverse to every $\gamma_i$, and let
$s_j = \sum_i t_i (\alpha \cap \gamma_i)$ where $\cap$ denotes signed intersection number.
Then $\area(C) = \sum_j s_j \cdot \area(R_j)$.

\vskip 12pt

We can now give necessary and sufficient conditions for a rational chain in a hyperbolic
surface to rationally bound an immersed subsurface.

\begin{proposition}\label{rot_is_scl}
Let $S$ be a compact, connected, oriented hyperbolic surface with geodesic boundary.
A rational chain $C$ in $B_1^H(S)$ rationally bounds a positive immersed subsurface if and only if
$$\scl(C) = \rot_S(C)/2$$
\end{proposition}
\begin{proof}
Given $C \in B_1^H$ and $f:T \to S$ a surface that it rationally bounds,
we can replace $T$ by a pleated surface (see e.g. Thurston \cite{Thurston_notes} Ch.~8 for an introduction to
the theory of pleated surfaces) and observe that the hyperbolic area of $T$ is at least
as big as $|\area(C)|$, with equality if and only if the map $f$ is an immersion
(note that if the original map was already an immersion, then so is its pleated representative).
It follows that an immersed surface is {\em extremal}, and therefore
by Theorem~\ref{bavard_duality_theorem}, if a chain $C$ bounds a positive immersed
surface, then $\scl(C) = \area(C)/4\pi$ by Gauss-Bonnet.

Conversely, since $S$ has boundary (by assumption) and therefore $\pi_1(S)$ is free,
any rational chain $C$ admits an extremal surface, by
the Rationality Theorem from \cite{Calegari_pickle} (or see \cite{Calegari_scl}, Thm.~4.18). 
Hence $\scl = \area(C)/4\pi$ if and only if $C$ bounds a positive immersed surface. By Lemma~\ref{area_rotation_number}
there is an equality $\area(C)/4\pi = \rot_S(C)/2$.
\end{proof}

In particular, the chain $\partial S$
satisfies
$$\scl(\partial S) = \area(S)/4\pi = -\chi(S)/2$$
Hence the surface $S$ itself is an extremal surface for $\partial S$.

\begin{remark}
We can use this fact to give a very short proof of Lemma~\ref{addition_is_transitive}
in the case that the ambient surface $S$ has boundary. A rational chain $C$ in $S$ rationally bounds
a positive immersed surface if and only if $\rot_S$ is extremal for $C$, i.e. if
$\scl(C) = \rot(C)/2$. If $C_1$ and $C_2$ rationally bound positive immersed surfaces,
then 
$$\scl(C_1 + C_2) \le \scl(C_1) + \scl(C_2) = \rot(C_1)/2 + \rot(C_2)/2 = \rot(C_1 + C_2)/2$$
Hence $\rot$ is extremal for $C_1 + C_2$, and therefore $C_1 + C_2$ rationally bounds
a positive immersed surface. Of course this proof is not ``really'' short, since it uses
the (highly nontrivial) fact that every rational chain in a free group bounds an extremal
surface.
\end{remark}

\begin{example}
Let $S$ be a once-punctured torus, with standard generators $a,b$. Let $w \in \pi_1(S)$
be a nontrivial commutator, and let $\gamma$ be the associated geodesic in $S$ (necessarily primitive).
It is a fact that in a free group, the ``standard'' once-punctured torus whose boundary is a given
nontrivial commutator is always extremal. When does $\gamma$ bound an immersed surface?
A description of $w$ as a cyclically reduced word in $a,b,a^{-1},b^{-1}$ determines a polygonal loop
$P_w$ in $\R^2$ with vertices contained in the integer lattice, as follows. Start at the origin,
and read the letters of $w$ one by one. On reading $a$ (resp. $a^{-1}$), take one step to the right
(resp. left), and on reading $b$ (resp. $b^{-1}$), take one step up (resp. down). The
polygonal loop $P_w$ can be ``smoothed'' by rounding the corners where a horizontal and a vertical
arc of $P_w$ meet, giving rise to an immersed loop $p_w$. It turns out that $\gamma$ bounds an
immersed surface in $S$ if and only if the winding number of $p_w$ is $\pm 1$. So for example,
$[a^n,b^m]$ bounds an immersed surface when $n,m \ne 0$, 
but $[a,ba^{-1}b^{-1}]$ does not. For details see \cite{Calegari_scl}, \S~4.2.
\end{example}

\subsection{Immersion theorem}\label{immersion_theorem_subsection}

We now prove our main technical result (Theorem~C) and deduce Theorem~A and Theorem~B as corollaries.

\begin{immerse_thm}
Let $S$ be a compact, connected orientable surface with $\chi(S)<0$, and $C = \sum r_i g_i$ a finite
rational chain in $B_1^H(\pi_1(S))$. Then for all sufficiently large rational
numbers $R$ (depending on $C$), the geodesic $1$-manifold in $S$ corresponding to the
chain $R\partial S + \sum r_i g_i$ rationally bounds
an immersed surface $f:T \to S$.
\end{immerse_thm}
\begin{proof}
Since $\partial S$ bounds the immersed surface $S$, it suffices to prove the theorem for
a particular positive $R$ (depending on $C$). 

Multiply through by a large positive integer to clear denominators, so we can assume
the $r_i$ are all integers. Also, by replacing $g_i$ with $g_i^{-1}$ if necessary, we
can assume the $r_i$ are all positive.
Pick a hyperbolic structure on $S$, and let $\gamma_i$ be the geodesic loop 
corresponding to the conjugacy class of $g_i$. 
By Scott (LERF for surface groups, \cite{Scott}) 
there is a finite cover of $S$ in which every component of the preimage of
each $\gamma_i$ is embedded. In other words, there is a finite cover
$\til{S} \to S$, so that if $\til{C}$ is the union in $\til{S}$ of all preimages of 
all components of $C$, then every component of $\til{C}$ is embedded (although the
union typically will not be).
The composition of a positive immersion with a covering
map is a positive immersion, so it suffices to construct the positive immersion in $\til{S}$. 
Hence without loss of generality we can assume that we are working in the cover, and
every individual geodesic $\gamma_i$ is embedded (though of course the union will typically not be).

Let $\alpha_1,\beta_1,\cdots,\alpha_g,\beta_g$ be a standard system of embedded geodesics which
are a standard basis for $H_1(S;\Z)/H_1(\partial S;\Z)$. 
For each $\gamma_i$ let $a_{i,j},b_{i,j}$ be integers such that
$$[\gamma_i] = \sum_j a_{i,j}[\alpha_j] + b_{i,j}[\beta_j] - D_i$$ in homology,
where $D_i$ is in the image of $H_1(\partial S;\Z)$.
Since the entire boundary $\partial S$ is homologically trivial, the class $D_i$
is represented (in many different ways) as a positive sum of positively oriented boundary
components of $S$.

The logic of the remainder of the argument is as follows. We will show that
for each $i$, the chain 
$$C_i:=\gamma_i - \biggl( \sum_j a_{i,j}\alpha_j - b_{i,j}\beta_j \biggr) + \partial_i$$
bounds a positive immersed surface, where $\partial_i$ is a positive sum of
positively oriented boundary components of $S$. 
This will be enough to prove the theorem.
For, since $C$ is homologically trivial, we have 
$$\sum r_i C_i = \sum_i r_i \gamma_i + R\partial S = C + R\partial S$$
Hence by Lemma~\ref{addition_is_transitive}, the chain $C + R\partial S$ bounds a positive
immersed surface for sufficiently big $R$, as claimed.

\vskip 12pt

First, decompose $S$ along a union $\delta$ of embedded separating geodesics
into a union of genus one subsurfaces $S_k$ such that
$\alpha_k,\beta_k$ is a standard basis for $H_1(S_k)/H_1(\partial S_k)$. 
Hence (in particular), for each $k$, the geodesics $\alpha_k,\beta_k$ are simple,
and intersect transversely in one point. The first step is to replace each $\gamma_i$ by
a union $\gamma_i''$ of geodesics, each of which is embedded and contained in
a single subsurface $S_k$.

\begin{claim}
For each $\gamma_i$ there is a chain $\gamma_i - \gamma_i'' + \partial_i$ which
bounds a positive immersed surface, where $\gamma_i''$ is a positive sum of embedded
geodesics disjoint from $\delta$, and $\partial_i$ is a positive sum of positively
oriented boundary components of $\partial S$. 
\end{claim}
We remark that either or both of $\gamma_i'',\partial_i$ might be empty in the statement of
the claim above.
\begin{proof}
If $\gamma_i$ is a component of $\delta$, then $\gamma_i$ cobounds a positive immersed
surface (in fact, an embedded subsurface of $S$)
together with some components of $\partial S$. After adding a sufficiently big multiple
of $\partial S$ we obtain a chain of the form $\gamma_i + \partial_i$ as above which bounds
an immersed positive surface, and we are done in this case.

Otherwise, $\gamma_i \cap \delta$ 
is in general position. The intersection $\gamma_i \cap S_k$ consists of a collection of arcs or a single
embedded loop. Moreover, since $\gamma_i$ and $\delta$ are geodesic, every arc
of $\gamma_i \cap S_k$ is essential. Among the components $S_k$, there are two which
intersect $\delta$ each in a single component. Let $S_k$ be one such. Every
arc of $\gamma_i \cap S_k$ has endpoints on this single component of $\delta$. There are
two possibilities (not necessarily mutually exclusive):
\begin{enumerate}
\item{There is an arc $\nu$ of $\gamma_i \cap \delta$ and an arc $\mu$ of $\delta$ on the
positive side of $\nu$ such that the interior of $\mu$ is disjoint from $\gamma_i$, and
$\partial \mu = \partial \nu$.}
\item{There are arcs $\nu,\nu'$ of $\gamma_i \cap \delta$ and arcs $\mu,\mu'$ of
$\delta$ on the positive sides of $\nu,\nu'$ such that the interiors of $\mu,\mu'$
are disjoint from $\gamma_i$, and $\partial \mu \cup \partial \mu' = \partial \nu \cup \partial \nu'$.}
\end{enumerate}
In the first case we build a positive immersed surface with one boundary component on
$\gamma_i$ by attaching a $1$-handle whose core is $\mu$. In the second case we build a
positive immersed surface with one boundary component on $\gamma_i$ by attaching a $1$-handle
whose core is one of $\mu,\mu'$. See Figure~\ref{normal_sum_figure}. 

\begin{figure}[htpb]
\labellist
\small\hair 2pt
\pinlabel $\nu$ at -20 75
\pinlabel $\mu$ at 80 75
\pinlabel $\delta$ at 50 140
\pinlabel $\text{resolve}$ at 190 95
\pinlabel $\nu$ at 540 95
\pinlabel $\nu'$ at 600 75
\pinlabel $\mu$ at 637 110
\pinlabel $\mu'$ at 640 40
\pinlabel $\delta$ at 610 145
\pinlabel $\text{resolve}$ at 755 95
\endlabellist
\centering
\includegraphics[scale=0.35]{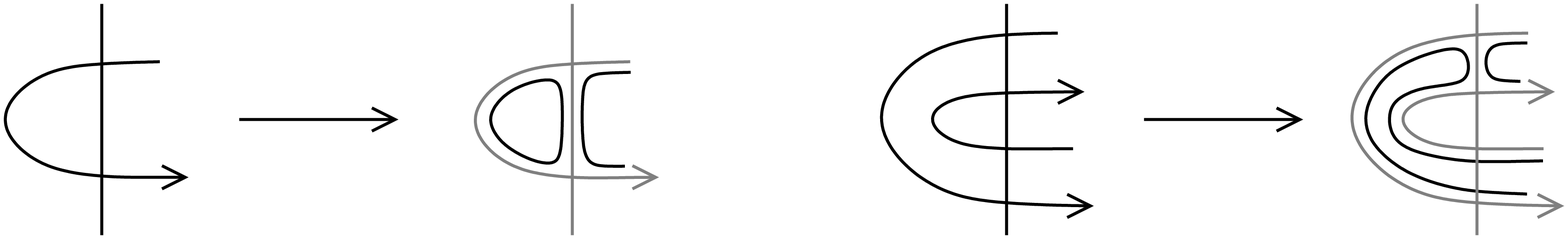}
\caption{The $1$-manifolds $\gamma_i$ and $-\gamma_i'$ cobound a 
positive immersed surface, obtained from $\gamma_i$ by attaching a $1$-handle whose core is $\mu$} \label{normal_sum_figure}
\end{figure}

The result of attaching a $1$-handle with core $\mu$ to a product neighborhood of $\gamma_i$
produces an embedded positive surface. If some component of the boundary of this surface is
homotopically trivial, it is necessarily trivial on the positive side, so we cap it off with an
embedded disk. Straighten the resulting surface by an isotopy until its boundary
components are geodesic. This produces a new geodesic
$1$-manifold $r(\gamma_i)$ called a {\em resolution} of $\gamma_i$
whose components are all embedded, such that $\gamma_i - r(\gamma_i)$
bounds a positive immersed surface, and such that $r(\gamma_i)$ has at least
two fewer intersections with $\delta$ than $\gamma_i$ does. In particular, it is certainly
true that each component of $r(\gamma_i)$ has fewer intersections with $\delta$ than
$\gamma_i$ does.

We resolve {\em the components} of $r(\gamma_i)$ in exactly the same way that we resolved
$\gamma_i$ above, and so on, inductively. If some of the resulting components are (isotopic to)
elements of $\delta$, they cobound an immersed positive surface with some positive
sum of positively oriented boundary components of $\partial S$ as above.

After finitely many steps, we obtain chains
$\gamma_i'$ and $\partial_i'$, where $\gamma_i'$ is a positive sum of embedded geodesics,
each disjoint from $\partial S_k$, where $\partial_i'$ is a positive sum of positively
oriented boundary components of $\partial S$, such that
$\gamma_i - \gamma_i' + \partial_i'$ bounds a positive immersed surface.

Now, let $S_{k'}$ be a component of $S-\delta$ adjacent to $S_k$. By construction, each
component of $\gamma_i'$ intersects at most one component of $\delta$ in $\partial S_{k'}$.
So the components of the $\gamma_i'$ can be iteratively resolved by the method above. By induction,
we end up (finally) with a chain $\gamma_i - \gamma_i'' + \partial_i$ of the desired form,
proving the claim.
\end{proof}

Let $\epsilon$ be a component of $\gamma_i''$ as in the claim. Let $S_k$ be the
component of $S-\delta$ containing $\epsilon$, and let $S'$ be obtained from $S_k$ by filling
in all but exactly one boundary component. For notational simplicity, denote the
geodesics $\alpha_k,\beta_k$ by $\alpha,\beta$ respectively, and let $a,b$ be integers such that 
$[\epsilon] = a[\alpha] + b[\beta]$ in $H_1(S')$.

Since $\epsilon$ is a geodesic on $S_k$ but not in $S'$, some component
of $S' - \epsilon \cup \alpha \cup \beta$
might be a bigon which contains one, or several components of $S' - S_k$.

By pushing $\epsilon$ repeatedly over such components of $S' - S_k$ we can eliminate such
bigons, innermost first. If $\epsilon'$ is the geodesic in $S_k$ obtained from $\epsilon$
by pushing $\epsilon$ over one component $\partial_i$ of $\partial S_k$, then $\epsilon$ and
$\epsilon'$ are disjoint, and either $\epsilon - \epsilon' + \partial_i$ or 
$\epsilon - \epsilon' + (\partial S_k - \partial_i)$ bounds a positive embedded subsurface of
$S_k$. By gluing up finitely many such surfaces, we obtain a geodesic $\epsilon''$ in $S_k$,
such that $\epsilon - \epsilon''$ plus some union of boundary components of $S_k$
bounds a positive immersed surface in $S_k$, and such that $\epsilon'' \cup \alpha \cup \beta$
is in the isotopy class of a configuration of geodesics for some hyperbolic structure on $S'$.
 
In this way we are reduced to arguing about embedded curves in a once-punctured torus.
We will show that $\epsilon'' - a\alpha - b \beta + n\partial S'$
rationally bounds a positive immersed surface in $S'$. If we can find such an immersed surface,
then by drilling out the components of $S'-S_k$ we will obtain a positive immersed surface in $S_k$ bounded by
$\epsilon'' - a\alpha - b\beta + \partial_\epsilon''$ for some suitable $\partial_\epsilon''$ which
is a positive sum of boundary components of $\partial S_k$.

By induction, and the (well-known) classification of simple curves in a once-punctured torus,
it suffices to show that the chain $ab - a - b + n[a,b]$ bounds an 
immersed surface in the once-punctured torus for sufficiently large $|n|$. 
In fact, this turns out to be true for $|n| \ge 2$.
Since the algebraic area in the once-punctured torus enclosed by $ab-a-b$ is $0$,
it suffices (by Lemma~\ref{area_rotation_number} and Proposition~\ref{rot_is_scl})
to show that $\scl(2[a,b] \pm (ab - a - b)) = \scl(2[a,b]) = 1$ which can
be verified by calculation, e.g. using {\tt scallop} (also see Figure~\ref{extremal_figure} below).

Applying Lemma~\ref{addition_is_transitive}, we conclude that
$\epsilon - a\alpha - b \beta + \partial_\epsilon$ rationally
bounds a positive immersed surface, where $\partial_\epsilon$ is some positive sum of
boundary components of $\partial S_k$. By adding on sufficient copies of $S-S_k$ we
obtain a positive immersed surface with boundary $\epsilon - a\alpha - b \beta + D_\epsilon$,
where $D_\epsilon$ is a positive sum of boundary components of $S$. But $\epsilon$ is
an arbitrary component of $\gamma_i''$. By the claim and Lemma~\ref{addition_is_transitive},
we obtain a positive immersed
surface with rational boundary $\gamma_i - \sum_j a_{i,j}\alpha_j - b_{i,j}\beta_j + \partial_i$
for suitable $\partial_i$. Since $i$ was arbitrary, this proves the theorem.
\end{proof}

\begin{remark}
Notice that we do {\em not} assume that $\partial S$ is nonempty, just that $\chi(S)<0$.
It is only the subsurfaces $S_k$ which are required to have nonempty boundary, which
will be the case, since each $S_k$ has genus $1$, and $\chi(S)<0$ implies that the genus
of $S$ is at least $2$.
\end{remark}

By the results of \S~\ref{rotation_subsection} and \S~\ref{duality_subsection} we conclude:

\begin{face_thm}
Let $F$ be a free group, and let $S$ be a compact, connected, orientable surface with
$\chi(S)<0$ and $\pi_1(S) = F$. Let $\partial S \in B_1^H(F;\R)$ be the $1$-chain represented
by the boundary of $S$, thought of as a finite formal sum of conjugacy classes in $F$.
Then the projective ray in $B_1^H(F;\R)$ spanned by $\partial S$ intersects the
unit ball of the $\scl$ norm in the interior of a face of codimension one in $B_1^H(F;\R)$.
\end{face_thm}

\begin{rot_thm}
Let $F$ be a free group, and let $S$ be a compact, connected, orientable surface
with $\chi(S)<0$ and $\pi_1(S) = F$. Let $\pi_S$ be the face of the $\scl$ unit norm ball
whose interior intersects the projective ray of the class $\partial S$. The face
$\pi_S$ is dual to the extremal homogeneous quasimorphism $\rot_S$.
\end{rot_thm}

\begin{remark}\label{explicit_surface_remark}
For the sake of completeness, we exhibit a positive immersed surface with rational
boundary $2[a,b] + a + b - ab$ in Figure~\ref{extremal_figure}.

\begin{figure}[htpb]
\labellist
\small\hair 2pt
\pinlabel $b$ at 60 210
\pinlabel $B$ at 60 162
\pinlabel $A$ at 58 123
\pinlabel $a$ at 58 78
\pinlabel $A$ at 135 200
\pinlabel $a$ at 155 153
\pinlabel $b$ at 155 130
\pinlabel $B$ at 135 85
\pinlabel $B$ at 230 162
\pinlabel $b$ at 225 210
\pinlabel $b$ at 215 249
\pinlabel $B$ at 202 295
\pinlabel $a$ at 400 260
\pinlabel $A$ at 375 217
\pinlabel $b$ at 330 207
\pinlabel $B$ at 340 162
\pinlabel $A$ at 340 123
\pinlabel $a$ at 330 78
\pinlabel $B$ at 375 69
\pinlabel $b$ at 400 30
\pinlabel $A$ at 230 123
\pinlabel $a$ at 225 79
\pinlabel $a$ at 215 35
\pinlabel $A$ at 202 -8
\endlabellist
\centering
\includegraphics[scale=0.35]{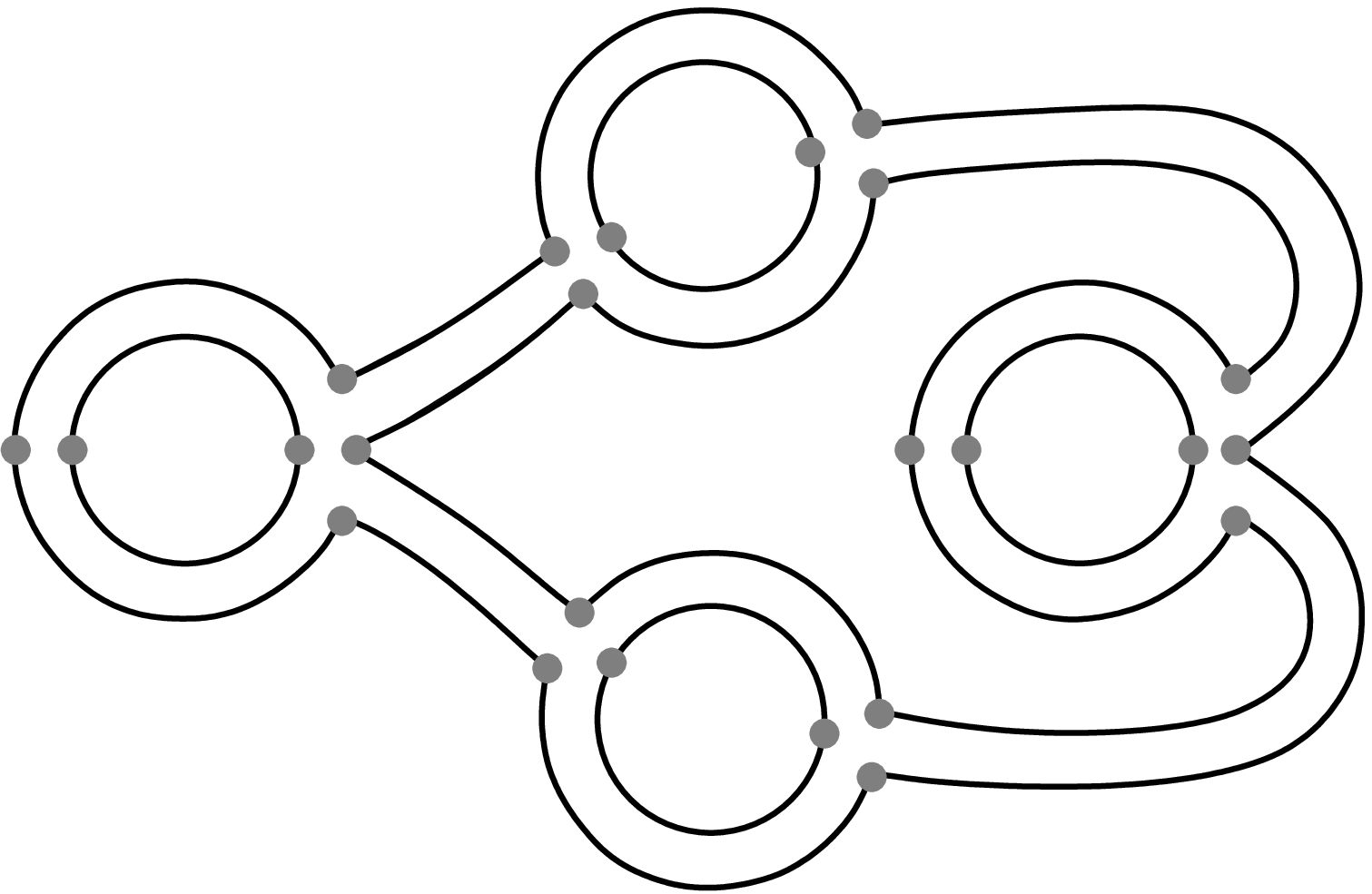}
\caption{A positive immersed surface with rational boundary the chain $2[a,b] + a + b - ab$} \label{extremal_figure}
\end{figure}
The figure depicts a genus $0$ surface $T$ with $6$ boundary components. The boundary components
are (cyclically) labeled by words in $F_2$ (for clarity, $A$ and $B$ are used in place of $a^{-1}$ and
$b^{-1}$). The components are decomposed into arcs each labeled by a letter, such that adjacent arcs
have opposite labels. It follows that there is a (unique) homotopy class of map $f:T \to S$
where $S$ is a once-punctured torus with standard generators $a,b$ for $\pi_1$ taking each boundary
component of $T$ to the geodesic in $S$ corresponding to the conjugacy class of the boundary label.
Two components of $\partial T$ each wrap twice around $[a,b]$ (the boundary of $S$). Two other components
of $T$ wrap once each around $(ab)^{-1}$. One component of $T$ wraps twice around $a$, and one
component wraps twice around $b$. Hence $\partial T$ represents the chain $2[a,b]^2 + a^2 + b^2 - 2ab$.
Since $-\chi(T)/2 = 2 = \scl(2[a,b]^2 + a^2 + b^2 - 2ab)$, the homotopy class of $f$ is represented by
an immersion.
\end{remark}

\subsection{Remarks and Corollaries}\label{remark_corollary_subsection}

In this section we collect some miscellaneous remarks and corollaries of our main theorems.
The first remark is that one can give a new proof of the relative version of
rigidity theorems of Goldman and Burger-Iozzi-Wienhard
(\cite{Goldman, Burger_Iozzi_Wienhard}; also compare Matsumoto \cite{Matsumoto}) about representations
of surface groups with maximal Euler class.

The context is as follows. Let $S$ be a compact oriented surface with boundary, and let
$\rho:\pi_1(S) \to \Sp(2n,\R)$ be a symplectic representation for which the conjugacy classes of
boundary elements fix a Lagrangian subspace. In this case, there is a well-defined relative
Euler class $\text{eu}_\rho$ in $H^2(S,\partial S;\Z)$ associated to $\rho$. It
is well-known in this context that $\text{eu}_\rho$ is a {\em bounded} cohomology class,
and satisfies $|\text{eu}_\rho([S])| \le -\chi(S)\cdot n$.

Our methods give a surprisingly short new proof of the following theorem 
(due to Goldman for $n=1$ and Burger-Iozzi-Wienhard for $n\ge 1$). For simplicity
we restrict to Zariski dense representations; this restriction can be removed by analyzing
various cases, but since
the main virtue of our alternate argument is its brevity, it is probably not worth spelling
out the details.

\begin{corollary}
Let $S$ be a compact oriented surface with boundary. Let
$\rho:\pi_1(S) \to \Sp(2n,\R)$ be Zariski dense, and suppose that boundary elements fix a Lagrangian
subspace (so that the relative Euler class $\text{eu}_\rho$ is defined). If
$\text{eu}_\rho([S])$ is maximal, $\rho$ is discrete.
\end{corollary}
\begin{proof}
In what follows, denote $\pi_1(S)$ by $G$ and its commutator subgroup by $G'$.
Since $S$ has boundary, $\text{eu}_\rho = \delta \phi_\rho$ where $\phi_\rho$ is
a homogeneous quasimorphism on $G$, unique up to elements of $H^1(S;\R)$. For
each $g \in G$, the value of $\phi_\rho(g)$ mod $\Z$ is the {\em symplectic rotation number}
of $\rho(g)$. The symplectic rotation number lifts to a quasimorphism on the universal cover
of $\Sp(2n,\R)$ with defect $n$.
On the other hand, $|\phi_\rho(\partial S)| = |\text{eu}_\rho([S])| = -\chi(S)\cdot n$ 
so we can conclude that the defect of $\phi_\rho$ on $G$ is exactly $n$, and $\phi_\rho$
is extremal for $\partial S$. Hence by Theorem~B we conclude that the symplectic
rotation number of every element of $\rho(G')$ is {\em zero}, and therefore (in particular)
$\rho(G')$ is not dense in $\Sp(2n,\R)$. Since $\Sp(2n,\R)$ is simple, every Zariski dense
subgroup is either discrete or dense (in the ordinary sense). If $\rho(G)$ is dense,
then the closure of $\rho(G')$ is normal in $\Sp(2n,\R)$. But $\Sp(2n,\R)$ is simple, and
the closure of $\rho(G')$ is a proper subgroup; hence $\rho(G)$ is discrete.
\end{proof}

Bavard \cite{Bavard} asked whether $\scl$ takes values in $\frac 1 2 \Z$ in a free group.
Though this turns out not to be the case, nevertheless, elements with values in $\frac 1 2\Z$
are very common. Theorem~C gives a flexible method to construct many
elements in free groups with $\scl$ in $\frac 1 2 \Z$. For example,
from Lemma~\ref{addition_lemma} we conclude:

\begin{corollary}
Let $F_2$ denote the free group on two generators $a,b$ and let $F_3 = F_2 * \langle c \rangle$
be a free group on three generators. For any $w \in [F_2,F_2]$ and for all integers $n$ with $|n|$
sufficiently large (depending on $w$) there is an equality
$$\scl_{F_3}([a,b]^ncwc^{-1}) = \frac {|n+\rot(w)|+1} 2 \in \frac 1 2 \Z$$
\end{corollary}
Obviously this construction can be varied considerably.

\begin{remark}
Computer experiments (using {\tt scallop}) suggest that for any $w \in [F_2,F_2]$, the geodesic
corresponding to the conjugacy class of $w[a,b]^n$ rationally bounds an immersed surface for sufficiently
large $n$. This can be proved directly for many specific elements $w$, but a general argument is lacking.
Therefore we make the following conjecture:
\begin{conjecture}
Let $w \in [F_2,F_2]$ be arbitrary. Then for sufficiently large integers $n$, there is an equality
$$\scl(w[a,b]^n) = \rot_S(w[a,b]^n])/2$$
where $\rot_S$ is the rotation quasimorphism associated with the realization of $F_2$ as
$\pi_1(S)$ where $S$ is a hyperbolic once-punctured torus.
\end{conjecture}
This certainly holds for many $w$. 
On the other hand, it is worth remarking that the projective classes
of such chains $w[a,b]^n$ are necessarily in the {\em boundary}
of the face $\pi_S$ of the $\scl$ norm ball.
\end{remark}

One can also deduce interesting corollaries for chains in closed
surface groups:

\begin{corollary}\label{injective_corollary}
Let $S$ be a closed, orientable surface with $\chi(S)<0$. Any
rational chain $C \in B_1^H(\pi_1(S);\R)$ rationally bounds an
injective surface.
\end{corollary}
\begin{proof}
By Theorem~C the chain $C$ rationally
bounds a positive immersed surface in $S$. But an immersion
between hyperbolic surfaces with geodesic boundary is necessarily
$\pi_1$-injective. 
\end{proof}

\begin{remark}
Theorem~C unfortunately does not settle the question
of whether $\scl$ is rational in closed (orientable) surface groups. A positive immersed
surface $T$ with rational boundary a chain $C \in B_1^H(S)$ is extremal in its
relative homology class, but the set of relative homology
classes with boundary $C$ is a torsor for the group $H_2(S;\Z)$, which is
infinite when $S$ is closed and orientable. One can translate this into
an absolute statement about $\scl$ at slight cost. 

Let $S$ be a closed hyperbolic surface,
and let $M$ be the unit tangent bundle of $S$.
The Euler class of $M$, thought of as an oriented 
circle bundle over $S$, is an element of $H^2(S;\Z)$ whose evaluation on
the fundamental class of $S$ is $\chi(S)$. The fundamental group of $M$ is
a central extension
$$0 \to \Z \to \pi_1(M) \to \pi_1(S) \to 0$$
corresponding to the Euler class in $H^2(S;\Z)$. Given $g \in [\pi_1(S),\pi_1(S)]$
there are many distinct lifts $\widehat{g}$ to $\pi_1(M)$ which differ by elements of
$\chi(S)\cdot\Z$. Theorem~C implies that for all but finitely
many lifts, $\scl_{\pi_1(M)}(\widehat{g}) \in \frac 1 2 \Z$ and $\widehat{g}$
is represented by a loop which rationally bounds an extremal surface
whose projection to $S$ is an immersion.
\end{remark}

The next corollary is a purely group-theoretic statement about
homologically trivial elements in closed surface groups.

\begin{corollary}\label{bounds_in_free_subgroup}
Let $G$ be a closed, orientable surface group with negative Euler characteristic,
and $g \in [G,G]$. Then there is a free subgroup $F$ of $G$ and elements
$g_i$ which are conjugates of positive powers of $g$ such that every $g_i \in F$,
and the product of the $g_i$ is in $[F,F]$.
\end{corollary}

\begin{remark}
Corollary~\ref{bounds_in_free_subgroup} can be thought of as saying that a homologically
trivial loop in a closed surface is covered by a homologically trivial $1$-manifold
in some surface of infinite index. One can ask to make this result sharper:
\begin{question}
Let $G$ be a closed, orientable surface group with negative Euler characteristic,
and $g \in [G,G]$. Is there a free subgroup $F$ of $G$ such that $g \in [F,F]$?
\end{question}
One can ask an analogous question for any group $G$. This question is especially
interesting when $G=\pi_1(M)$ where $M$ is a closed $3$-manifold.
\end{remark}

Finally, we can use Theorem~C to construct injective closed
surface groups representing homology classes in certain graphs of groups.
The following is the analogue of the main theorem from \cite{Calegari_surface}.

\begin{corollary}
Let $G$ be a graph of free and closed orientable surface groups with $\chi<0$ amalgamated
along cyclic subgroups, and let $A$ be a homology class in $H_2(G;\Q)$.
Let $S_i$ be the closed surface vertex subgroups, and 
$r_1, . . , r_n$ any rational numbers with all $|r_i|$ sufficiently large.
Then the class $A + r_1S_1 + r_2S_2 + . . + r_nS_n$
is rationally represented by a closed surface subgroup of $G$.
\end{corollary}
\begin{proof}
This follows immediately by the argument of Thm.~3.4 from \cite{Calegari_surface}, and
Theorem~C.
\end{proof}

It is natural to wonder whether every rational chain in $B_1^H(F)$, where $F$ is a free
group, is projectively contained in the interior of a face of the $\scl$ norm of finite codimension, but in
fact this is not the case, as the following example shows.

\begin{example}\label{infinite_codimension}
By Bavard duality, the codimension of the face whose interior contains the projective
class of a rational chain $C$ is one less than the dimension of the space of
extremal quasimorphisms for $C$ (mod $H^1$). Hence to exhibit a rational chain (in fact, an
element of $[F,F]$) which is in the interior of a face of infinite codimension,
it suffices to exhibit a chain which admits an infinite dimensional space of extremal
quasimorphisms.

Let $F = F_1 * F_2$ where $F_1$ and $F_2$ are both free of rank at least $2$, and let
$g \in [F_1,F_1]$ be nontrivial. Let $\phi_1 \in Q(F_1)$ be extremal for $g$, and let $\phi_2 \in Q(F_2)$
be arbitrary with $D(\phi_2)\le D(\phi_1)$. By the Hahn-Banach theorem, there exists
$\phi \in Q(F)$ which agrees with $\phi_i$ on $F_i$, and satisfies $D(\phi) = D(\phi_1)$.

Another (more direct) way to see that $g$ is contained in a face of infinite codimension
is as follows. Let $h$ be a nontrivial element of $[F_2,F_2]$. 
Let $X_i$ be a $K(F_i,1)$ (e.g. a wedge of two circles) for $i=1,2$. Let $X$ be a $K(F,1)$
obtained by joining $X_1$ to $X_2$ by an edge $e$. Let $f:S \to X$ be extremal for the chain
$g^n + h$. If $S$ maps over $e$, we can compress $S$ along the preimage of a generic point on $e$,
reducing $-\chi^-(S)$, which is absurd since $S$ is extremal. Hence $S$ consists of two surfaces
$S_1,S_2$ one of which is extremal for $g^n$ in $X_1$, and one of which is extremal for $h$ in $X_2$.
Hence 
$$\scl_F(ng + h) = \scl_{F_1}(g^n) + \scl_{F_2}(h) = \scl_{F_1}(g^n) + \scl_{F_2}(h^{-1}) = \scl_F(ng - h)$$
so $g$ is not in the interior of a top face of the $\scl$ norm restricted to the
vector subspace of $B_1^H(F)$ spanned by $g$ and $h$.
\end{example}

\section{Acknowledgment}

I would like to thank Marc Burger, Benson Farb, David Fisher, \'Etienne Ghys,
Bill Goldman, Walter Neumann, Cliff Taubes and Anna Wienhard for their comments.
I am also very grateful to the anonymous referee for corrections and
thoughtful comments. While writing this paper I was partially funded by NSF grant DMS 0707130.


\begin{thebibliography}{99}
\bibitem{Bavard}
	C. Bavard,
	\emph{Longueur stable des commutateurs},
	Enseign. Math. (2), {\bf 37}, 1-2, (1991), 109--150
\bibitem{Blank}
	S. Blank,
	\emph{Extending immersions of the circle},
	Ph. D. thesis, Brandeis 1967.
\bibitem{Burger_Iozzi_Wienhard}
	M. Burger, A. Iozzi and A. Wienhard,
	\emph{Surface group representations with maximal Toledo invariant},
	preprint, arXiv:math.DG/0605656
\bibitem{Calegari_quasi}
	D. Calegari,
	\emph{Universal circles for quasigeodesic flows},
	Geom. Topol. {\bf 10} (2006), 2271--2298
\bibitem{Calegari_pickle}
	D. Calegari,
	\emph{Stable commutator length is rational in free groups},
	preprint, arXiv:0802.1352
\bibitem{Calegari_surface}
	D. Calegari,
	\emph{Surface subgroups from homology},
	Geom. Topol. {\bf 12} (2008), 1995--2007
\bibitem{Calegari_scl}
	D. Calegari,
	\emph{scl},
	monograph, to be published in MSJ monograph series;
	available from 
	
	\quad\quad\quad{\tt http://www.its.caltech.edu/$\sim$dannyc}
\bibitem{scallop}
	D. Calegari,
	{\tt scallop},
	computer program, available from 
	
	\quad\quad\quad{\tt http://www.its.caltech.edu/$\sim$dannyc}
\bibitem{Dantzig}
	G. Dantzig,
	\emph{Linear Programming and Extensions},
	Princeton Univ. Press, Princeton, 1963
\bibitem{Francis}
	G. K. Francis,
	\emph{Spherical curves that bound immersed discs},
	Proc. AMS. {\bf 41}, No. 1, (1973), 87--93
\bibitem{Goldman}
	W. Goldman,
	\emph{Discontinuous groups and the Euler class},
	Ph. D. thesis, UC Berkeley 1980.
\bibitem{Gromov_bounded}
	M. Gromov,
	\emph{Volume and bounded cohomology},
	Inst. Hautes \'Etudes Sci. Publ. Math. No. 56 (1982), 5--99
\bibitem{Grothendieck}
	A. Grothendieck,
	\emph{Sketch of a Programme},
	available from 
	
	\quad\quad\quad{\tt http://people.math.jussieu.fr/$\sim$leila/grothendieckcircle/mathtexts.php}
\bibitem{Kronheimer_Mrowka}
	P. Kronheimer and T. Mrowka,
	\emph{Scalar curvature and the Thurston norm},
	Math. Res. Lett. {\bf 4} (1997), 931--937
\bibitem{Maclane}
	S. Mac Lane,
	\emph{Homology},
	Springer classics in mathematics, Berlin, 1995
\bibitem{Matsumoto}
	S. Matsumoto,
	\emph{Some remarks on foliated $S^1$ bundles},
	Invent. Math. {\bf 90} (1987), no. 2, 343--358
\bibitem{Milnor_doodle}
	J. Milnor,
	\emph{A concluding amusement: symmetry breaking},
	in {\em Collected Papers, volume III: Differential Topology}, AMS 2007, 315--317
\bibitem{Oszvath_Szabo}
	P. Oszvath and Z. Szabo,
	\emph{Link Floer homology and the Thurston norm},
	J. Amer. Math. Soc. {\bf 21} (2008), 671--709
\bibitem{Scott}
	P. Scott,
	\emph{Subgroups of surface groups are almost geometric},
	J. London Math. Soc. (2) {\bf 17} (1978), no. 3, 555--565
\bibitem{Thurston_norm}
	W. Thurston,
	\emph{A norm for the homology of $3$-manifolds},
	Mem. Amer. Math. Soc. {\bf 59} (1986), no. 339, i-vi and 99-130
\bibitem{Thurston_notes}
	W. Thurston,
	\emph{Geometry and Topology of $3$-manifolds (a.k.a. ``Thurston's Notes'')},
	notes from lectures at Princeton University; available from the MSRI
\end{thebibliography}
\end{document}